\documentclass[12pt]{amsart}
\usepackage{mathrsfs}
\usepackage{amsthm,amsmath}
\usepackage{xypic}
\usepackage{rotating}
\usepackage{setspace}
\usepackage{pstricks}
\usepackage{graphicx}
\usepackage{enumerate}

\xyoption{all}

\usepackage{textcomp}
\usepackage{amsfonts}
\usepackage{amssymb}
\usepackage{amsthm}
\usepackage{stmaryrd}

\newcommand{\nin}{\not \in}

\newcommand{\reals}{\mathbb{R}}
\newcommand{\naturals}{\mathbb{N}}

\newcommand{\integers}{\mathbb{Z}}

\newcommand{\inv}[1]{{#1}^{-1}}

\newcommand{\kerm}{{\textrm{ ker }}}

\newcommand{\autm}{\textrm{Aut}}

\newcommand{\boundaryinf}{\partial}

\newtheorem*{corollary*}{Corollary}
\newtheorem*{theorem*}{Theorem}
\newtheorem{theorem}{Theorem}
\newtheorem{lemma}{Lemma}

\newtheorem{proposition}[lemma]{Proposition}
\newtheorem{corollary}{Corollary}
\newtheorem{observation}[lemma]{Observation}

\theoremstyle{definition}
\newtheorem{definition}{Definition}
\newtheorem{example}{Example}
\newtheorem*{example*}{Example}

\theoremstyle{remark}
\newtheorem{remark}{Remark}

\newcommand{\ds}{\displaystyle}

\newcommand{\metricball}[3]{\textrm{Ball}_{#1}(#2, #3)}

\title[Controll Connectivity for Semi-Direct Products...]{Controlled Connectivity for Semi-Direct Products Acting on Locally Finite Trees}
\author{
Keith Jones \\ 
}

\begin{document}

\begin{abstract}

In 2003 Bieri and Geoghegan generalized the Bieri-Neuman-Strebel
invariant $\Sigma^1$ by defining $\Sigma^1(\rho)$, $\rho$ an isometric action by a
finitely generated group $G$ on a proper CAT(0) space $M$. In this paper,
we show how the natural and well-known connection between Bass-Serre
theory and covering space theory provides a framework for the calculation of
$\Sigma^1(\rho)$ when $\rho$ is a cocompact action by $G = B \rtimes A$,
$A$ a finitely generated group,  on a locally finite Bass-Serre tree $T$ for $A$.
This framework leads to a theorem providing conditions for
including an endpoint in, or excluding an endpoint from,
$\Sigma^1(\rho)$. When $A$ is a finitely generated free group acting
on its Cayley graph, we can restate this theorem from a more
algebraic perspective, which leads to some general results on
$\Sigma^1$ for such actions.
\end{abstract}

\maketitle

\section{Introduction}

In \cite{bierigeoghegansl2}, the Bieri and Geoghegan begin with the following:

\begin{quote}

Given a group $G$ and a contractible metric space $M$, consider the
set Hom($G$, Isom($M$)) of all actions by $G$ on $M$ by isometries.
Are there invariants of such actions which distinguish one from
another? Are there topological properties which one such action might
possess while another might not?

\end{quote}

The tool they apply to draw distinctions between such actions is controlled 
$n$-connectivity, which is developed in \cite{amsmemoir}, and which we briefly
describe here.  Suppose $\rho$ is an isometric action by a group $G$ having type 
$F_n$ on a proper CAT(0) metric space $M$. Fixing a basepoint $b\in M$, the CAT(0) 
boundary, $\boundaryinf M$, can be thought of as the set of geodesic
rays $\tau$ emanating from $b$.\footnote{For background on the
topological finiteness property ``type $F_n$'', see \cite[\S
7.2]{geogheganbook}, and for background on CAT(0) metric spaces and
their boundaries see \cite[II.1 and II.8]{bridsonhaefliger}. A metric space is {\em
proper} if each closed metric ball is compact.} For an end point $e
\in \boundaryinf M$ represented by a ray $\tau$, there is a nested
family of subsets $HB_k(\tau)$, $k \in \reals$, called ``horoballs'' which serve as 
metrics balls (in $M$) ``at $e$''.\footnote{For background on
horoballs, see \cite[\S10.1]{amsmemoir}. The convention followed
there and in this paper is that as $k$ increases, we approach $e$, the
reverse of the convention in \cite{bridsonhaefliger}.} This provides a sense
of direction in $M$, which can be ``lifted to $G$'' by $\rho$
via a $G$-equivariant map from the $n$-skeleton
of the universal cover of a $K(G,1)$. For a point $e \in \boundaryinf
M$, if the lifts of the horoballs about $e$ are (roughly)
$(n-1)$-connected, then we say $\rho$ is controlled $(n-1)$-connected
over $e$.
Bieri and Geoghegan show that this is independent of choice of
$K(G,1)$ or equivariant map. The invariant $\Sigma^n(\rho)$ is the subset of
$\boundaryinf M$ consisting of points over which $\rho$ is controlled
$(n-1)$-connected.
This definition generalizes the Bieri-Neumann-Strebel-Renz (BNSR)
invariants $\Sigma^n(G)$, which are open subsets of the
CAT(0) boundary of the vector space $G_{ab} \otimes
\reals$. A key difference between $\Sigma^1(\rho)$ and the BNSR
invariant $\Sigma^1(G)$ is that $\Sigma^1(\rho)$ is in general {\em
not} an open subset of $\boundaryinf M$.

Apart from enabling one to draw geometric
distinctions between isometric actions by a group on a proper CAT(0)
space, the invariant can also provide group theoretical information:
if the orbits under an action $\rho$ are discrete, then the point stabilizers
are finitely generated if and only if $\Sigma^1(\rho) = \boundaryinf
M$ \cite[Theorem A and Boundary Criterion]{amsmemoir}.

When $M = T$ is a locally finite (simplicial) tree, the CAT(0) boundary is
a metric Cantor set.  Initial results in
\cite{amsmemoir} led the authors to ask whether in this case
$\Sigma^1(\rho)$ might
always be one of $\emptyset$, a singleton, or the entire boundary
$\boundaryinf T$. Work in \cite{kjones} establishes a class of actions
for which this is the case. However, work by Ralf Lehnert in his diploma
thesis demonstrates that other subsets of $\boundaryinf T$ can be realized
as $\Sigma^1(\rho)$ for certain actions \cite{lehnert}. This hints at
a potentially rich world of $\Sigma^1$ invariants, which we further
explore here. 

\subsubsection*{Acknowledgments} This paper extends results from 
the author's Ph.D dissertation at Binghamton University. It
would not have been possible without the continuous support and
encouragement of Ross Geoghegan. The author also wishes to thank 
Robert Bieri for his generosity with his time in discussing some of
the ideas involved in this work. Additionally, the author in indebted to the
anonymous referee, whose efforts have lead to a more interesting and clearer
paper.

\subsection{Statement of results}
We restrict our attention to $\Sigma^1$, and study only the following
scenario:

\begin{definition}[Actions of interest]
\label{myactions}
Let $A$ be a finitely generated group with finite generating set $R$, and let 
$T$ be a locally finite\footnote{A simplicial tree is a proper metric
space iff it is locally finite.} simplicial tree on which $A$ acts
cocompactly and with finitely generated stabilizers. 
For a group $B$, suppose we have a homomorphism
$\varphi: A \rightarrow \autm(B)$, 
and let $G = B \rtimes_\varphi A$ be the
resulting semidirect product. Elements of $G$ are of the form
$(b,a) \quad a \in A,\ b \in B$ 
and multiplication in $G$ operates under the rule
\[ (b_1, a_1)(b_2, a_2) = (b_1 a_1b_2a_1^{-1}, a_1a_2) =
(b_1\varphi_{a_1}(b_2), a_1a_2). \]

Suppose $G$ is finitely generated. Then it follows that $B$ is
finitely generated as an $A$-group. By this, we mean there is a finite subset 
$S \subset B$ such that the set $\{\varphi_a(S) \mid a \in A\}$
generates $B$, and so $G$ is generated by $S \cup R$.

The natural projection $G \twoheadrightarrow A$ induces an action
$\rho$ by $G$ on $T$
which contains the normal subgroup $B$ in its kernel. We investigate
$\Sigma^1(\rho)$.
\end{definition}

\begin{remark}
\label{nonfgstabilizersremark}
As mentioned, if the point stabilizers under $\rho$ are finitely
generated, then $\Sigma^1(\rho) = \boundaryinf T$ 
\cite[Theorem A and Boundary Criterion]{amsmemoir}. Moreover, since $T$ is
locally finite, all point stabilizes are commensurable, so if any one is
finitely generated, then all are. Thus with the assumption
that the stabilizers under the $A$ action on $T$ are finitely
generated, in order to obtain ``interesting'' invariants (those with
$\Sigma^1(\rho) \neq \boundaryinf T$), one must  assume
that $B$ is not finitely generated, since the stabilizers under $\rho$
are simply semidirect products of $B$ with the stabilizers in $A$. 
\end{remark}

%

\subsubsection{Main Result} 
With the action $\rho: G \rightarrow Isom(T)$ as defined above,  we
apply the relationship between Bass-Serre theory and covering space
theory to construct a commutative diagram of $G$-equivariant cellular 
maps between CW-complexes:
\[ \xymatrix{ 
\tilde X \ar[d]^p \ar[r]^{\tilde r}  &  T \ar[d]^{id} \\
\bar X \ar[d]^q \ar[r]^{\bar r} & T \ar[d]^{mod\ G} \\
X \ar[r]^r & V = T \backslash G}\]
where $X$ is a $K(G,1)$, $\bar X$ is a $K(B,1)$, $\tilde X$ is a
contractible universal cover, $p$ and $q$ are covering
projections, and $r$, $\bar r$, and $\tilde r$ are retracts.\footnote{This is the topological construction of the
Bass-Serre tree \cite[\S 6.2]{geogheganbook}, \cite{scottandwall},
discussed further in Section \ref{bstheorysec}.} For a geodesic ray
$\tau$ in $T$ and $k \in
\integers$, consider the horoball $HB_k(\tau)$.\footnote{A precise
description of $HB_k(\tau)$ is given in Equation \eqref{horoballdef} in
subsection \ref{consequencessubsection}.} 
 For  $W \subset X$ a finite subcomplex, set 
\[\bar X_{(\tau, k, W)} = \inv{\bar r}(HB_k(\tau)) \cap \inv{q}(W)
\subset \bar X.\]

\begin{theorem}
\label{Sigma1bymappingtreethm}  
Let $e \in \boundaryinf T$ be represented by a geodesic ray $\tau$.

\begin{enumerate}[(i)]
\item If there exists a finite subcomplex $W \subset X$ 
such that for every $k \in \integers$,  $\bar X_{(\tau, k, W)}$
is connected and the map on $\pi_1$ induced by the inclusion $\bar
X_{(\tau, k, W)} \hookrightarrow \bar X$ is surjective,
then $e \in \Sigma^1(\rho)$.

\item If for every $k \in \integers$ 
and every finite subcomplex $W \subset X$ such that $\bar X_{(\tau, k, W)}$
is connected, the induced map on $\pi_1$ is not surjective, then $e \nin \Sigma^1(\rho)$.
\end{enumerate}
\end{theorem}

\subsubsection{Consequences and examples}
\label{consequencessubsection}
Theorem \ref{Sigma1bymappingtreethm} has a number of consequences in the case
where $A$ is a free group and $T$ is its Cayley graph. In this case, the
 vertices
of $T$ are the elements of $A$. Let $e$ be an endpoint of
$T$ and suppose the geodesic ray $\tau$ represents $e$. For an integer
$k$, let $A_k(\tau)$ be the elements of $A$ that form the vertex set
of the horoball $HB_k(\tau)$.  To avoid confusion with the
group $B$, we will
use the notation $\metricball{r}{X}{p}$ to refer to the metric ball of radius
$r$ in the space $X$ about the point $p$. Just as the horoball
$HB_k(\tau)$ can be written as the nested union of closed metric balls
in $T$:
\begin{equation}
\label{horoballdef}
 HB_k(\tau) = \bigcup_{l \geq \max\{0,k\}}
\overline{\metricball{l-k}{T}{\tau(l)}},
\end{equation}
the set $A_k(\tau)$ can be written as a nested union of closed metric
balls in the word metric on $A$:
\begin{equation}
\label{ahorodef}
A_k(\tau) = \bigcup_{l \geq \max\{0,k\}}
\overline{\metricball{l-k}{A}{\tau(l)}}.
\end{equation}

We will say $B$ is finitely generated over a
{\em subset} $A' \subseteq A$ if there is a finite subset $S \subseteq B$
such that $\{as\inv{a} \mid s \in S, a \in A' \}$ generates
$B$. \\

In Section \ref{freegpsection}, we show that in this context 
Theorem \ref{Sigma1bymappingtreethm} can be restated as follows:
\begin{theorem}
\label{Sigma1bygeneratorsthm} 
Let $A$ be a finitely generated free group, and let $T$ be its Cayley
graph with respect to a free basis. For the action $\rho$ as in
Theorem \ref{Sigma1bymappingtreethm}, and for $e \in \boundaryinf T$
represented by geodesic ray $\tau$, 
\begin{enumerate}[(i)]
\item If there is a finite set $S \subseteq B$ such that for each $k \in
\integers_{\geq 0}$, $S$ generates $B$ over $A_k(\tau)$, then $e \in
\Sigma^1(\rho)$.

\item If for each $k \in \integers_{\leq 0}$, $B$ is not finitely
generated over $A_k(\tau)$, then $e \nin \Sigma^1(\rho)$.

\end{enumerate}
\end{theorem}

This is reminiscent of the invariant $\Sigma_B(A)$ of
\cite{BieriNeumannStrebel} and
\cite{bieristrebelvaluations}, but whereas $\Sigma_B(A)$ is determined by
the algebraic structure of $G$, our sets $A_k(\tau)$ are given by the
geometry of $T$; in particular, they are not monoids. 

Since $B$ is finitely generated over $A$, we have:

\begin{corollary}
\label{Sigma1bysameimagecor}
If for each $k \in \integers_{\geq 0}$, $\varphi(A_k(\tau)) =
\varphi(A)$, then $e \in \Sigma^1(\rho)$.
\end{corollary}


Let $\{a_1,\dots,a_n\}$ freely generate $A$. For a generator
$a_i$, let the function $expsum_{a_i}$ map a reduced word $w$ in
$\{a_1,\dots,a_n\}^\pm$ to the
corresponding exponent sum of $a_i$ in $w$, and define the function
$expsum_{a_i^{-1}}$ to be $-expsum_{a_i}$.

\begin{corollary}
\label{boundedexpsumcor}
Let $t \in \{a_1,\dots,a_n\}^\pm$. 
Suppose there does not exist an $m \in \integers$ such that $B$
is finitely generated over $A - expsum_t^{-1}([m, \infty))$. 
(I.e., any subset $A' \subseteq A$ must have reduced words with arbitrarily 
large exponent sum of $t$ in order for $B$ to be finitely generated over 
$A'$.) Then any endpoint represented by a word eventually consisting of only
$t^{-1}$ does not lie in $\Sigma^1(\rho)$.
\end{corollary}

\begin{example}
\label{lehnertexample}
This is a generalization of an example calculated by Ralf
Lehnert, although the methods used here are different from
his.  
Consider the semidirect product $G = B \rtimes A$, where $B =
\integers[\frac{1}{p_1p_2\dots p_n}]$, where $p_i \neq p_j$ are prime
for $1\leq i,j \leq n$, and $A$ is free on
$\{a_1,\dots,a_n\}$. The action is given by $a_i$ acting by multiplication by
$\frac{1}{p_i}$. For $A' \subseteq A$,  $B$ is finitely generated over
$A'$ if and only if $A'$ contains reduced words with arbitrarily large
exponent sum of each $a_i$.  
One can show that for any $k \in \integers$, this will always be the case
for $A' = A_k(\tau)$ unless $\tau$ eventually consists of only $a_i^{-1}$ 
(see Lemma \ref{expsumlemma}).
Thus, by Corollary \ref{boundedexpsumcor}, any endpoint corresponding to an
infinite word eventually consisting of $a_i^{-1}$ for some $i$ is not
in $\Sigma^1$.  By Theorem \ref{Sigma1bygeneratorsthm}, any other endpoint is in
$\Sigma^1$.
\end{example}

\begin{example} 
Let $G = \integers \wr
\integers = \oplus_{i \in \integers} \langle b_i \rangle \rtimes
\langle t \rangle$. The action is by shifting: $^tb_i = b_{i+1}$. Let
$T$ be the Cayley graph of $\langle t \rangle$, a simplicial line. The action
$\langle t \rangle \curvearrowright T$, induces an action $G
\stackrel{\rho}{\curvearrowright} T$. It is known from
previous work that $\Sigma^1(\rho)$ is empty, as follows. 
Because the endpoints of the action are fixed, we can relate
$\boundaryinf T$ to homomorphisms $G \twoheadrightarrow \integers$, and an end point
lies in $\Sigma^1(\rho)$ iff the corresponding homomorphism represents a point
of the BNSR invariant $\Sigma^1(G)$ \cite[\S 10.6]{amsmemoir}. These homomorphisms
do not represent points of
$\Sigma^1(G)$ because they are not homomorphisms associated to HNN
extension decompositions of $G$ over finitely generated base groups
\cite[Prop. 3.1]{browntrees}. Here it follows from
Theorem \ref{Sigma1bygeneratorsthm}, because $B$ is not finitely generated
over any proper subset of $\langle t \rangle$. 
\end{example}

Corollary \ref{Sigma1bysameimagecor} can be applied to determine a
nice criterion for finding endpoints of $T$ lying in
$\Sigma^1(\rho)$.

\begin{theorem}
\label{subwordsactingtriviallythm} With notation as in Theorem
\ref{Sigma1bygeneratorsthm}, viewing endpoints of $T$ as infinite words in the generators of $A$,
$\Sigma^1(\rho)$ contains any endpoint represented by an infinite word 
containing infinitely many mutually distinct subwords lying in $\kerm \varphi$.
\end{theorem}

\begin{corollary} 
\label{abelianactioncor}
If $\varphi(A) \leq \autm(B)$ is abelian and $A$ has rank $n \geq 2$, then
$\Sigma^1(\rho)$ is nonempty.
\end{corollary}

For example, any endpoint represented by an infinite word containing
infinitely many commutators will be contained in $\Sigma^1(\rho)$.

\begin{example}
\label{freeandtrivialexample}
For positive integers $m \geq n$, let 
$C = \langle a_1, \dots, a_n\rangle$ and $D = \langle
a_{n+1}, \dots, a_m\rangle$ be free groups, and $A = C * D$.
For a finitely generated group $K$, let $G$ be the restricted wreath product 
$K \text{ wr}_C\ A$, where the $A$-action on the indexing set $C$
is defined by the composition of the natural projection $\pi: A
\rightarrow C$ and left multiplication. In other words, 
 $G = B \rtimes_\varphi
A$, where $\ds B = \oplus_{\omega \in C} K_\omega$ with each $K_\omega$
a copy of $K$. The elements of $B$ are
sequences $(x_\omega)$, $x_\omega \in K_\omega$, $\omega \in C$, with only 
finitely many $x_\omega$ nontrivial,  and $C$
acts on $B$ by permuting the indices (by left multiplication on
itself) while $D \leq \text{ker }
\varphi$. The projection  $G \twoheadrightarrow A$ followed by the
natural action by $A$ on its Cayley graph 
$T = \Gamma(A, \{a_1,a_2,\dots, a_m\})$
induces an action $\rho$ on $T$.

By Theorem \ref{subwordsactingtriviallythm}, any endpoint containing
infinitely many letters $a_i^\pm$, $n < i \leq m$ will lie in $\Sigma^1(\rho)$, while
Theorem \ref{boundedexpsumcor} ensures that any endpoint eventually
consisting of a single letter $a_j^\pm$, $1 \leq j \leq n$, will not lie in
$\Sigma^1(\rho)$. In fact, any end point represented by a geodesic ray that
eventually consists of only letters from $C$ lies outside
$\Sigma^1(\rho)$, as is argued in Section \ref{freeandtrivialargument}.
So an endpoint lies in $\Sigma^1(\rho)$ if and only if a
representative geodesic ray contains
infinitely many letters from $D$.
\end{example}

\begin{example}
One can also perform calculations in the case where $A$ is not free.

For example, let $H$ be any finitely generated group and consider the group
\begin{equation}
G = B \rtimes_\varphi A, \quad B = \prod_{i \in \integers} H. \quad
A= <a \mid a^4> * <b \mid b^4>,
\end{equation}

where $\varphi: A \rightarrow Aut(B)$ consists first of the projection
onto $D_\infty$ collapsing $a^2$ and $b^2$ to the identity, followed
permutation of the indices $i \in \integers$ given by
by the natural action by $D_\infty$ on $\integers$. 
Let $\rho$ be the
action by $G$ on the regular 4-valent Bass-Serre tree $T_4$ corresponding to
the free product structure of $A$.  Notice, since this is
the Bass-Serre tree corresponding to a free product, any point 
$e \in \boundaryinf T_4$ corresponds to a word in the normal form for
the free product. 
One can apply Theorem \ref{Sigma1bymappingtreethm}
to calculate $\Sigma^1(\rho)$ directly to determine that a given $e
\in \boundaryinf T_4$ if and only if it corresponds to an infinite
normal form word containing infinitely many subwords of the form $a^2$
or $b^2$. 

There is a stark similarity between this result and Theorem 
\ref{subwordsactingtriviallythm}, and indeed a statement similar to
Theorem \ref{subwordsactingtriviallythm} can be made in the case where
$A$ is a free product. However, only when $A$ is a free product of
{\em finite} groups will its corresponding Bass-Serre tree be locally
finite (and hence proper); and in this case the Kurosh subgroup theorem implies
 that $A$ has a free subgroup $A'$ of finite
index. If $G = B \rtimes A$, then $G' = B \rtimes A'$ is a finite
index subgroup of $G$, and the action $\rho$ by $G$ on the Bass-Serre 
tree corresponding to the free product decomposition of $A$ restricts
to an action by $G'$ on the same tree. It follows from \cite[Theorem
12.1]{amsmemoir} that the invariant is the same for both actions.
Hence, it is not clear that such an endeavor will add anything new to
the discussion.

\end{example}

\subsection{Defining $\Sigma^1$}
\label{sigmaonedefinitionsection}

In general, there is a family of invariants $\Sigma^n$, $n\geq 0$,
corresponding to the notion of controlled $(n-1)$-connectivity 
\cite{amsmemoir}. The
discussion below refers only to $\Sigma^1$ and controlled
connectivity, but a similar discussion can be had in full generality.

In \cite{amsmemoir}, Bieri and Geoghegan introduced:

\begin{definition}[Controlled Connectivity -- Original]
\label{origcontrolledconnectivitydef}
Let $\rho$ be an action by a finitely generated group $G$ on a proper CAT(0)
metric space $(M,d)$.  Choose a $K(G,1)$ complex $X$ whose
universal cover $\tilde X$ has a cocompact $1$-skeleton $(\tilde X)^{(1)}$,
and a continuous $G$-map $h: (\tilde X)^{(1)} \rightarrow M$.
Given a geodesic ray $\tau$ in $M$, $\tau(\infty)$ denotes the point 
of $\boundaryinf M$ represented by $\tau$.  For $t \in \reals$, let 
$\tilde X_{(\tau,t)}$ denote the largest subcomplex contained in 
$\inv{h}(HB_t(\tau))$.  Then $h$ is {\em controlled connected over
$\tau(\infty)$} if there exists $\lambda: \reals \rightarrow
[0,\infty)$ such that for all $t \in \reals$,
any two points of $\tilde X_{(\tau,t)}$ can be connected by a
path in $\tilde X_{(\tau,t-\lambda(t))}$ and $t-\lambda(t) \rightarrow
\infty$ as $t\rightarrow \infty$.
\end{definition} 

In \cite[p. 143]{bierigeoghegansl2},
they provide an ``extended'' definition, which we will
show coincides with Definition \ref{origcontrolledconnectivitydef} 
when $G$ is finitely generated.

\begin{definition}[Controlled Connectivity -- ``Extended'']
\label{controlledconnectivitydef}
Let $\rho$ be an action by a group $G$ (not necessarily finitely
generated!) on a proper $CAT(0)$
metric space $(M,d)$. Choose a non-empty free contractible $G$-CW-complex 
$\tilde X$, and a continuous $G$-map $h: \tilde X \rightarrow M$. 
Fix a geodesic ray $\tau$ in $M$.
For $t \in \reals$, define $\tilde X_{(\tau, t)}$ to be the largest
subcomplex of $h^{-1}(HB_t(\tau))$. Then $h$ is {\em controlled
connected over $\tau(\infty)$} if for every cocompact $G$-subspace $\tilde W
\subseteq \tilde X$, there exists a cocompact $G$-subspace $\tilde W'$
containing $\tilde W$ such that for all $t \in \reals$, there exists
$\lambda(t) \geq 0$ satisfying: 

\begin{itemize}

\item[($*$)] Any two points of $\tilde X_{(\tau, t)} \cap \tilde W$
can be connected by a path through $\tilde X_{(\tau, t-\lambda(t))} \cap
\tilde W'$
\item[($**$)] Any two points of $\tilde X_{(\tau, t + \lambda(t))}
\cap \tilde W$ can be connected by a path through $\tilde X_{(\tau,
t)} \cap \tilde W'$.
\end{itemize}

\end{definition}

Both Definitions \ref{origcontrolledconnectivitydef} and 
\ref{controlledconnectivitydef} are independent of choice of
$G$-space $\tilde X$ or $G$-map $h: \tilde X \rightarrow M$, as is
proved in
\cite{amsmemoir} and \cite{bierigeoghegansl2}, respectively, in what
the authors commonly refer to as the ``Invariance Theorem.'' For
Definition \ref{controlledconnectivitydef}, this 
is proved for the related concept of controlled connectivity over  
$a\in M$ \cite[Theorem 2.3]{bierigeoghegansl2}, and the authors point out
 (on p. 143) that this proof carries over to
controlled connectivity over an end point. 

The parameter $\lambda(t)$ is called a {\em lag}. 
In nice cases, $\lambda$ may be constant, or even 0.
A lag is necessary for invariance, but an
arbitrarily generous lag would defeat the point.
In Definition \ref{controlledconnectivitydef}, condition $(**)$
effectively replaces the condition that $t - \lambda(t) \rightarrow \infty$
found in Definition \ref{origcontrolledconnectivitydef}.

Suppose now that $G$ is finitely generated 
and $h: \tilde X \rightarrow M$ satisfies Definition
\ref{controlledconnectivitydef}, but $\tilde X$ has
non-cocompact $1$-skeleton. There is
$h':\tilde X'\rightarrow M$, where $\tilde X'$ has cocompact
$1$-skeleton, which by the Invariance Theorem also satisfies
Definition \ref{controlledconnectivitydef}.  We now show that 
Definition \ref{origcontrolledconnectivitydef} is satisfied by
$h'|_{(\tilde X)^{(1)}}$.

\begin{proposition}
\label{definitionscoincideprop}
Let $G$ be a finitely generated group, $\tilde X$ a contractible free 
$G$-complex with cocompact $1$-skeleton $(\tilde X)^{(1)}$, and  
geodesic ray $\tau$ in a proper CAT(0) space $M$. A $G$-map $h: \tilde X \rightarrow M$ satisfies
Definition \ref{controlledconnectivitydef}  iff the restriction $h|: (\tilde
X)^{(1)} \rightarrow M$ satisfies Definition
\ref{origcontrolledconnectivitydef}. 
\end{proposition}
\begin{proof}

If $h|$ satisfies Definition \ref{origcontrolledconnectivitydef} over
 $\tau(\infty)$, then there is a lag $\lambda(t)$
satisfying $t - \lambda(t) \rightarrow \infty$ as $t \rightarrow
\infty$ such that for each $t$, any two points in $(\tilde
X)^{(1)}_{(\tau,t)}$ may be joined in $(\tilde X)^{(1)}_{(\tau,
t-\lambda(t))}$. 
Let $\tilde W$ be any cocompact $G$-subset of $\tilde X$. Let $Y$ be
the smallest subcomplex of $\tilde X$ containing $\tilde W$. Then $Y$ is
still a cocompact $G$-set.  Take $\tilde W' = Y \cup (\tilde X)^{(1)}$. 
Then any two points of $\tilde X_{(\tau, t)} \cap \tilde W$ may
be joined in $\tilde X_{(\tau, t- \lambda(t))} \cap \tilde W'$  
by first moving into the $1$-skeleton of $\tilde X_{(\tau, t)} \cap Y$.
We now replace $\lambda(t)$ with a lag function $\lambda'(t)$ satisfying
 both $(*)$ and $(**)$.  For any $t$, there exists $r>t$ such that for
all $s \geq r$, $s - \lambda(s) > t$. (So points of 
$(\tilde X)^{(1)}_{[\tau,s)}$ can be connected through a path in $(\tilde
X)^{(1)}_{[\tau, t)}$.)
Let $\lambda'(t) = \max\{\lambda(t), r-t\}$. 

Now suppose $h$ satisfies Definition \ref{controlledconnectivitydef} 
over $\tau(\infty)$.  For $\tilde W = \tilde W' = (\tilde X)^{(1)}$, 
there is $\lambda:\reals\rightarrow [0,\infty)$ such that 
by $(*)$, 
any two points of $\tilde X_{(\tau,t)} \cap (\tilde X)^{(1)}$ may
 be joined through a path in $\tilde X_{(\tau,t-\lambda(t))} \cap (\tilde
X)^{(1)}$, since a path may be chosen which does not leave $(\tilde X)^{(1)}$.
We now find a lag $\lambda'(t)$ satisfying $t - \lambda'(t) \rightarrow \infty$.
Since $HB_s(\tau) \subseteq HB_r(\tau)$ when $s>r$, $(**)$ says that 
for all $r \in \reals$, for all $t > r + \lambda(r)$, a lag of $(t-r)$
suffices for $HB_t(\tau)$. Hence, we may choose a real-valued sequence 
$s_1 < s_2 < \dots$ satisfying $s_n \rightarrow \infty$ and for 
$t \in [s_n, s_{n+1})$ a lag of $t - n$ suffices. 
Define $\lambda'(t)$ by:
\[ \lambda'(t) = \left\{\begin{array}{ccl} 
\lambda(t) & \text{ if } & t < s_1 \\
t - n & \text{ if } & s_n \leq t < s_{n+1},\ n = 1,2,\dots \\
\end{array}\right.\]
Then $ t -  \lambda'(t) = n$ when $s_n \leq t < s_{n+1}$, so 
$t - \lambda'(t) \rightarrow \infty$ as $t\rightarrow \infty$.
\end{proof}

This means that one may test for controlled connectivity of a finitely
generated group in the traditional sense by applying the more general
definition with a space $\tilde X$, even when $(\tilde X)^{(1)}$ is
not cocompact.

\begin{definition}[$\Sigma^1$]
The Invariance Theorem ensures controlled connectivity is a property
of the action $\rho$, so we define
\[\Sigma^1(\rho) = \{ e \in \boundaryinf M
\mid \rho \textrm{ is controlled connected over } e\} \]
\end{definition}

The action $\rho$ induces an action on $\boundaryinf M$, and under
this action $\Sigma^1(\rho)$ is a $G$-invariant set.

\section{Covering Spaces and Bass-Serre Theory}
\subsection{Some facts about covering spaces.}
\label{coveringspacesec}

The following proposition counts the number of components over a
connected subset in a covering projection. 

\begin{proposition}[Theorem 3.4.10 of \cite{geogheganbook}]
\label{preimagedisconnectedprop}
Let $(X, Z)$ be a pair of path connected CW complexes, both containing
a point $z$. Let $i: (Z, z) \rightarrow (X, z)$ be the inclusion
map, and let $p: (\bar X, \bar z) \rightarrow (X, z)$ be a covering
projection. Let $H_1 = im\ p_\#$ and $H_2 = im\ i_\#$. Then the number of
path components of $p^{-1}(Z)$ equals the order of the set of double cosets
\[ \{H_1gH_2 \mid g \in \pi_1(X,z) \}. \]

In particular, if $\bar X = \tilde X$ is the universal cover of $X$, then the
number of components of $p^{-1}(Z)$ is the index of $H_2$ in
$\pi_1(X,z)$.
\end{proposition}

For us the interesting case for us will be when $Z$ has
connected preimage in $\tilde X$. With this in mind, we will say $Z$ is
{\em $\pi_1$-surjective} when the inclusion $(Z,z) \hookrightarrow
(X,z)$ induces a surjection on $\pi_1$.

A second fact we will need is a consequence of path lifting:

\begin{proposition}
\label{componentsurjectlemma}
Let $(X,Z)$ be a pair of path connected CW complexes. Let $p: \bar X
\rightarrow X$ be a covering projection. Then each component of
$p^{-1}(Z)$ surjects onto $Z$.
\end{proposition}

\subsection{Bass-Serre theory via covering spaces}
\label{bstheorysec}
We are concerned with cocompact actions by finitely generated
groups on locally finite simplicial trees, particularly those without
global fixed points. Thus all actions we consider can be
understood though Bass-Serre theory \cite{basscoveringtheory},
\cite{serretrees}. There is a beautiful connection between Bass-Serre
theory and Covering Space theory \cite[\S 6.2]{geogheganbook} and
\cite{scottandwall}, which we take advantage of
in order to calculate $\Sigma^1$ for actions as described by
Definition \ref{myactions}.
Here we briefly recount this topological construction of the Bass-Serre
tree in the context of such actions, and in the process introduce 
an intermediary covering space which will be important for calculations.

Given an action $\rho$ as in Definition \ref{myactions}, set 
$V = G \backslash T$, a finite graph since $\rho$ is cocompact.
 Fix a base vertex $v_0$ of $V$. 
Choose a connected fundamental domain $F$ for $\rho$, and let $\mathcal V$
be the system of stabilizers for $F$. (Here a
fundamental domain is not a subgraph if $V$ has loops.) Let $\bar v_0$
be the vertex of $F$ over $v_0$.
Let $\mathbb{V} = (V,
\mathcal V, v_0)$ be the corresponding graph of groups associated with
$\rho$.

For a cell (vertex or edge) $c$ of $V$, the stabilizer $G_c \in
\mathcal V$ is of the form $B \rtimes A_c$ (where $A_c \leq A $ is the
stabilizer of $c$ under the action by $A$). Following Remark
\ref{nonfgstabilizersremark}, we assume $G_c$ is not finitely
generated. Let $R_c$ be a finite generating set for $A_c$, and let
$S_c$ be an infinite generating set of $B$ which contains a finite set $S$
such that $S$ generates $B$ over $A$, as described in Definition
\ref{myactions}. 
Let $X_c$ be a $K(G_c, 1)$-complex having a single 0-cell and 1-cells
in correspondence with $R_c \cup S_c$ \cite[Ch. 7]{geogheganbook};
this is called a ``vertex (or edge) space,'' depending on whether $c$
is a vertex or edge.
There is covering space $\bar X_c \twoheadrightarrow X_c$ which is
a $K(B,1)$, since $B \leq G_c$. 

As in \cite[Theorem 7.1.9]{geogheganbook}, 
we assemble a $K(G,1)$-complex $(X,x_0)$ as a total space for the
graph of groups $(V,\mathcal V, v_0)$. This is formed as a disjoint
union of the vertex spaces $X_v$, to which we attach
$X_e \times I$ for each edge $e$. 
The attaching maps are such that the induced maps on $\pi_1$ induce
inclusions $G_e \hookrightarrow G_v$ when $v$ is an endpoint of $e$.
 There is a retraction $r: (X,
x_0) \rightarrow (V, v_0)$ collapsing $X_c$ (or $X_c\times I$ if $c$
is an edge) to $c$ for each cell $c$
of $V$. There is a covering space $q: (\bar X, \bar
x_0)\rightarrow (X, x_0)$ corresponding to $B$. This, too, can be
described as a total space of a graph of groups where the graph is the
tree $T$ itself, and each stabilizer is isomorphic to $B$, since $T =
B\backslash T$. 
 
We then have the universal cover
$p:(\tilde X, \tilde x) \twoheadrightarrow (\bar X, \bar x)$. Above
the map $r$, are maps $\bar r: (\bar X, \bar x_0) \rightarrow (T, \bar
v_0)$ and $\tilde r: (\tilde X, \tilde x_0) \rightarrow (T, \bar
v_0)$.

All maps are $G$-equivariant and continuous. 
We arrive at the commutative diagram given before the statement of Theorem
\ref{Sigma1bymappingtreethm}.

\section{Analysis of $\Sigma^1$ via subcomplexes of $\bar X$}
\label{analysissec}
We continue using the notation of the previous section.

\begin{remark}
\label{useintegersremark}
Let the end point $e$ be represented by the geodesic ray $\tau$. 
Because $\tau$ emanates from a vertex, the horoball $HB_t(\tau)$ is a
subtree of $T$ if and only if $t \in \integers$. We are interested in 
$\tilde X_{(\tau, t)} \subset \tilde X$, which is by definition the largest subcomplex of $\inv{(\bar r \circ
p)}(HB_t(\tau))$; and by choice of $\tau$, $X$, $\bar r$, and
$p$, $\tilde X_{(\tau, t)} = \inv{(\bar r \circ p)}(HB_t(\tau))$ exactly
when $t \in \integers$. (There are no 0-cells of $\bar X$
mapped by $\bar r$ to the interior of an edge of $T$.) 
Hence, it is enough to look at horoballs
of the form $HB_k(\tau)$, $k \in \integers$. Similarly, the lag
$\lambda$ can always be taken to be in $\integers$, so that all
horoballs under consideration are subtrees of $T$.

\end{remark}

\begin{definition}
\label{suitablesubcomplexdef}
A finite subcomplex $W$ will be called {\em suitable} if for each
subtree $U$ of $T$, the set $\inv{\bar r}(U) \cap \inv{q}(W)\subset
\bar X$
is connected. By Remark \ref{useintegersremark}, it follows that if
$W$ is suitable, then for
any horoball $HB_k(\tau)$, the set 
$\bar X_{(\tau,k,W)} = \inv{\bar r}(HB_k(\tau)) \cap  \inv{q}(W)$ is
connected.
\end{definition}

\begin{lemma}
\label{intersectionwithhoroballlemma}
Suppose $W$ is a connected subcomplex of $X$ such that for each vertex
$v$ of $F$, $W$ contains the 1-cells of $X_v \subset X$ corresponding
to $R_v$. Moreover, for each edge $e$ of $F$, let $x_e \in X_e$ be the
basepoint, and suppose $W$ contains the 1-cell 
$\{x_e\} \times [0,1] \in X$.  Then $W$ is suitable. 
\end{lemma}

\begin{proof}
Let $U$ be a subtree of $T$. We show that $q^{-1}(W)
\cap \bar r^{-1}(U)$ is connected.  For a given vertex $v$ of $F$, 
$W$ contains loops generating $A_v$, and the image of the map $\bar
X_v \hookrightarrow X_v$ is $B$. By Proposition
\ref{preimagedisconnectedprop} (with $H_1 \geq A_v$ and $H_2 = B$),
$q^{-1}(W) \cap \bar X_v$ is connected. Hence the lemma
holds if $U$ is any vertex of $T$. If $U$ contains edges, then since
$W$ contains all edges of $X$ corresponding to
base points of $X_e$, $e\in F$, there must be a path in
$q^{-1}(W) \cap \bar r^{-1}(U)$ from the $\bar r$-preimage of any one
vertex of $U$ to any other. Furthermore, the fact that there is no
cell of $\bar X$ lying completely over the interior of an edge of $T$ ensures
that there can be no components of $q^{-1}(W) \cap \bar r^{-1}(U)$
over the interior of an edge. 
\end{proof}

Because each stabilizer $A_v$ is finitely generated and $V$ is finite,
the following observation follows from Lemma
\ref{intersectionwithhoroballlemma}.
\begin{observation}
\label{nicecomplexlemma}
If $W \subseteq X$ is compact, then there exists a suitable subcomplex
$W' \subseteq X$ such that $W \subseteq W'$.
\end{observation}

For convenience, we restate Theorem \ref{Sigma1bymappingtreethm}
before proving it. Recall that $\bar X_{(\tau, k, W)}$ denotes 
$\inv{\bar r}(HB_k(\tau)) \cap \inv{q}(W) \subset \bar X$.\\\

\noindent
{\bf Theorem \ref{Sigma1bymappingtreethm}. }{\em
Let $e \in \boundaryinf T$ be represented by a geodesic ray $\tau$.

\begin{enumerate}[(i)]
\item If there exists a finite subcomplex $W \subset X$ 
such that for every $k \in \integers$,  $\bar X_{(\tau, k, W)}$
is connected and the map on $\pi_1$ induced by the inclusion $\bar
X_{(\tau, k, W)} \hookrightarrow \bar X$ is surjective,
then $e \in \Sigma^1(\rho)$.

\item If for every $k \in \integers$ 
and every finite subcomplex $W \subset X$ such that $\bar X_{(\tau, k, W)}$
is connected, the induced map on $\pi_1$ is not surjective, then $e \nin \Sigma^1(\rho)$.
\end{enumerate}}

\begin{proof}[Proof of Theorem \ref{Sigma1bymappingtreethm}]
\ \

\begin{enumerate}[(i)]
\item 
We show that Definition \ref{controlledconnectivitydef} is satisfied with lag $\lambda = 0$;
in this case, conditions $(*)$ and $(**)$ are the same.
Let $\tilde L \subseteq \tilde X$ be a cocompact $G$-subcomplex and
set $L = q(p(\tilde L))$. Let $k \in \integers$. By Observation
\ref{nicecomplexlemma}, 
there is a suitable subcomplex $W' \subseteq X$ with $L \cup W
\subseteq W'$. Since $\bar X_{(\tau, k, W)}$ is $\pi_1$-surjective
onto $\bar X$, it follows that $\bar X_{(\tau, k, W')}$ is as well.
Because $W'$ is suitable,
Proposition \ref{preimagedisconnectedprop} applies to $\bar X_{(\tau,
k, W')} \subset \bar X$ to ensure that
$\inv{p}(\inv{q}(W')) \cap \tilde X_{(\tau, k)}$ is connected.
 Moreover this contains $L \cap
\tilde X_{(\tau, k)}$, so condition $(*)$ is satisfied.

\item Let $\tilde L$ be a cocompact $G$-subcomplex of $\tilde X$, and
let $\tilde L'$ be any cocompact $G$-subcomplex of $\tilde X$
containing $\tilde L$. We show that for any lag $k \geq 0 \in
\integers$, there exist points of $\tilde L \cap \tilde X_{(\tau, 0)}$
lying in distinct components of $\tilde L' \cap \tilde X_{(\tau,
-k)}$. 

Let $L = p(q(\tilde L))$ and $L' = p(q(\tilde L'))$. By Observation
 \ref{nicecomplexlemma}
 there exists a suitable complex $W \subseteq X$ with $L' \subseteq
W$.  Then $\bar X_{(\tau, -k,
W)}$ is connected, and by assumption it is not $\pi_1$-surjective.
Set $\tilde W = \inv{q}(\inv{p}(W))$. Then $\tilde W \cap \tilde
X_{(\tau, -k)}$ is disconnected by Proposition
\ref{preimagedisconnectedprop}. Furthermore, Proposition
\ref{componentsurjectlemma} ensures that each of its components
contains components of $\tilde L' \cap \tilde X_{(\tau, -k)}$, which in
turn contain points of $\tilde L \cap \tilde X_{(\tau, -k)}$.

\end{enumerate}
\end{proof}

\section{$A$ a free group}
\label{freegpsection}
Let the action $\rho$ by $G$ on $T$ be as defined in Definition
\ref{myactions}, with the additional restriction that $A$ is a free
group on the set $\{a_1,\dots, a_n\}$ and $T$ is its Cayley graph
with respect to this set.
Then the vertices of $T$ are the elements of $A$. 
Let $X$, $q: (\bar X, \bar x_0) \rightarrow (X,\bar x)$, and $p:
(\tilde X, \tilde x_0) \rightarrow (\bar X, \bar x_0)$,
$r: X \rightarrow V$, and $\bar r: \bar X \rightarrow T$ be as defined
in Section \ref{bstheorysec}.  The graph $V = A \backslash
T$ has a unique vertex $v_0$, so the $K(G,1)$-complex $X$ can be
chosen to have a unique 0-cell $x_0$, which we naturally choose as
basepoint for $X$.  In this case, for any cell $c$ of $V$, $X_c$ and
$\bar X_c$ are both $K(B,1)$-complexes. In fact, we can take $\bar X_c
= X_c = X_{v_0}$ for all $c$, since passing from $X$ to $\bar X$ simply 
``unwraps'' loops in
$A \subseteq G = \pi_1(X, x_0)$.  Choose the base point $\bar x_0$ of
$\bar X$ to be the unique 0-cell of $\bar X$  mapped to $1 \in A =
\textrm{vert } T$. 

We uniquely represent $\boundaryinf T$ by geodesic
rays $\tau$, with $\tau(0) = 1 \in A$ and $\tau(n)$  a freely reduced
word on $n$ letters. Thus each geodesic ray $\tau$ corresponds to a
unique infinite freely reduced word $\prod_{i \in \integers_{\geq 0}} c_i$.

\subsection{From suitable complexes to subgroups}

From here on, we identify $B$ with $\pi_1(\bar X, \bar x_0)$.
Let $W$ be a suitable subcomplex of $X$. Since $W$ is finite, the
subgroup
\[ B(W) = inclusion_\#(q^{-1}(W) \cap \bar r^{-1}(1), \bar x_0) \leq B \]
is finitely generated. Let $S(W)$ be a finite generating set for
$B(W)$.  Let $T'$  be a subtree of $T$.
Fix $v \in \text{vert } T' \subseteq A$. Then $\bar r^{-1}(v)
\cong \bar X_c$ has a single 0-cell; call it $x'$.
Let $B(W,T',v)$ be the image of \mbox{$\pi_1(q^{-1}(W)
\cap \bar r^{-1}(T'), x')$} in $\pi_1(\bar X, x')$.  Let 
\[\Psi_v: \pi_1(\bar X, x') \rightarrow \pi_1(\bar X, \bar x_0) =B\]
be the change of basepoint isomorphism. Then for $g\in \pi_1(\bar X, x')$,
 $\Psi_v(g) = vgv^{-1}$.

\begin{lemma}
\label{subtreesubgrouplemmafree}
$\Psi_v(B(W, T', v))$ is the subgroup of $B$ generated by 
$\{ usu^{-1} \mid s \in S(W), u \in T' \}$.
\end{lemma}

\begin{proof}
Any element $h \in B(W, T', v)$ can represented by a loop $\sigma_h$ in 
the $1$-skeleton of $q^{-1}(W) \cap \bar r^{-1}(T')$ based
at $x'$. Because $\bar X$ has no $0$-cells over the interiors of
edges of $T$, and because each vertex space is a copy of $X_{v_0}$ and
each edge space a copy of $X_{v_0} \times [0,1]$,
the loop $\sigma_h$ may be decomposed as concatenation of
subpaths $\sigma_h^0, \sigma_h^1, \dots \sigma_h^m$, $m \in
\naturals$, where each
$\sigma_h^i$ $0 \leq i \leq m$ is either a 1-cell joining one vertex
space to another (a ``base edge'' for an edge space) or a loop contained
entirely in a vertex space and corresponding to some $s\in S(W)$.
Between each pair of subpaths, we may introduce a path which returns
straight back to $x'$ (i.e., via 1-cells over lying over edges of $T'$
exclusively). This
process rewrites $h$ as a product of conjugates of the form
$v^{-1}usu^{-1}v$, $s \in S(W)$, $u \in T'$.
\end{proof}

Combining Theorem \ref{Sigma1bymappingtreethm} with Lemma
\ref{subtreesubgrouplemmafree}, we obtain a purely algebraic 
condition for determining whether an endpoint lies in $\Sigma^1(\rho)$. 
For a geodesic ray $\tau$ corresponding to the infinite word
$\prod_{i} c_i$ and $k \in \integers$, define $A_k(\tau) = 
\text{vert}(HB_k(\tau))$ and $w_k = \tau(k) = c_1c_2\dots c_k$. 
Then 
\[\pi_1(q^{-1}(W) \cap \bar r^{-1}(HB_k(\tau)), w_k) = B(W,
HB_k(\tau), w_k). \]

\noindent {\bf Theorem 2.} {\em 
Let $A$ be a finitely generated free group, and let $T$ be its Cayley
graph with respect to a free basis. For the action $\rho$ as in
Theorem \ref{Sigma1bymappingtreethm}, and for $e \in \boundaryinf T$
represented by geodesic ray $\tau$, 
\begin{enumerate}[(i)]
\item If there is a finite set $S \subseteq B$ such that for each $k \in
\integers_{\geq 0}$, $S$ generates $B$ over $A_k(\tau)$, then $e \in
\Sigma^1(\rho)$.

\item If for each $k \in \integers_{\leq 0}$, $B$ is not finitely
generated over $A_k(\tau)$, then $e \nin \Sigma^1(\rho)$. 
\end{enumerate}
}

\begin{proof}[Proof of Theorem \ref{Sigma1bygeneratorsthm}]
\begin{enumerate}[(i)]
\item If there is such a finite set $S$, then we can choose a suitable
subcomplex $W$ containing loops corresponding to $S$. For any $k \in
\integers_{\geq 0}$, let $x'$ be the unique vertex of $\inv{\bar r}(w_k)$,
 and we have that 
\[B(W, HB_k(\tau), w_k) = \Psi^{-1}_{w_k}(B) = \pi_1(\bar X, x').\]
Thus by Theorem \ref{Sigma1bymappingtreethm}, part (i), $e \in
\Sigma^1(\rho)$.

\item Given a suitable subcomplex $W$ of $X$ and $k \in
\integers_{\leq 0}$, the subgroup
$\Psi(B(W, HB_k(\tau), w_k))$ is by assumption a proper subgroup of
$B$. Hence, $B(W, HB_k(\tau), w_k)$ is a proper subgroup of
$\pi_1(\bar X, x')$. Thus by Theorem \ref{Sigma1bymappingtreethm},
part (ii), $e \nin \Sigma^1(\rho)$. 

\end{enumerate}
\end{proof}


Recall that for $t \in \{a_1,\dots,a_n\}^{\pm}$, the function $expsum_{t}$
maps a reduced word $w$ in $\{a_1,\dots,a_n\}^\pm$ to the
corresponding exponent sum of $t$ in $w$. Also, recall we use the
notation $Ball_{r}(A,v)$ to refer to the $r$-ball around $v$ in $A$
(in the word metric), to avoid confusion with the subgroup $B$.

\begin{lemma}
\label{expsumlemma}
For an endpoint $e$ represented by the geodesic ray $\tau$, 
let $Q_{t,k}(\tau) = \{expsum_t(v) \mid v \in A_k(\tau)\} \subseteq \integers$. 
Then $Q_{t,k}(\tau)$ is bounded above iff $\tau$ eventually consists of only
$t^{-1}$.  Moreover, $Q_{t,k}(\tau)$ contains every integer within its bounds.
\end{lemma}

\begin{proof}
Let $\tau$ be represented by the infinite word $c_1c_2\dots$, and fix
$k \in \integers$. 
Recall that $A_k(\tau) = \cup_{l \geq \max\{0,k\}}
\overline{Ball_{l-k}(A, c_1c_2\dots c_l)}$.

Suppose for $N \in \integers$, $c_i = t^{-1}$ for all $i > N$.
For $j = 0,1,2,\dots$, the words 
$g_j = c_1c_2\dots c_{N+j}t^{N+j-k}$ all represent the same element of
$A$, and $g_j$ has maximal $expsum_t$ among elements of
$\overline{Ball_{N+j-k}(A, c_1c_2\dots c_{N+j})}.$
Since $A_k(\tau)$ is the union of these subsets, it follows 
that $Q_{t,k}(\tau)$ is bounded above.

On the other hand, suppose that there are  infinitely many $i \in
\integers$ such
that $c_i \neq t^{-1}$.  For $j \in \integers$, $j \geq \max\{0,k\}$,  let $m(j)$ be the
number of letters $c_i$ in $c_1c_2\dots c_j$ with $c_i \neq t^{-1}$. By 
assumption $m(j) \rightarrow \infty$ as $j \rightarrow \infty$.
Let $g_j = c_1\dots c_j t^{j-k}$. Then 
\[ g_j \in  \overline{Ball_{j-k}(A, c_1c_2\dots c_j)} \subseteq
A_k(\tau).\]
Since \mbox{$expsum_t(c_1c_2\dots c_j) \geq -(j - m(j))$},
\[expsum_t(g_j)  =  expsum_t(c_1c_2\dots c_j) + {j-k}  \geq  m(j) - k. \]
Letting $j\rightarrow \infty$, we have that $Q_{t,k}(\tau)$ is not
bounded above.

The fact that $Q_{t,k}(\tau)$ contains every integer within its
bounds follows from the observation that for $v,\ w \in A_k(\tau)$,
if 
\[expsum_t(v) < m < expsum_t(w),\] 
the path connecting $v$ to $w$
contains a vertex $u$ with \mbox{$expsum_t(u) = m$.}
\end{proof}

\begin{proof}[Proof of Corollary \ref{boundedexpsumcor}]
Let $t \in \{a_1,\dots,a_n\}^\pm$. Suppose $e \in \boundaryinf T$ is
represented by an infinite word eventually consisting of only
$t^{-1}$, and suppose there exists no  $m \in \integers$ such that $B$
is finitely generated over $A - expsum_t^{-1}([m,\infty))$. 
By Lemma \ref{expsumlemma}, $\{expsum_t(a) \mid a \in
A_k(\tau)\}$ is bounded above. Hence, $B$ cannot be finitely generated
over $A_k(\tau)$, and so by Theorem \ref{Sigma1bygeneratorsthm}, part
$(ii)$, $e \nin \Sigma^1(\rho)$.
\end{proof}


\begin{proof}[Proof of Theorem \ref{subwordsactingtriviallythm}]
Let $e = \tau(\infty)$, with $\tau$ corresponding to the infinite word
$\prod_i c_i$. By Corollary \ref{Sigma1bysameimagecor},
it is enough to show that for each $k \geq 0 \in \integers$, 
$\varphi(A_k(\tau)) =\varphi(A)$.

Let $w \in \mathcal A^*$ be a freely reduced word, and let $l$ be the
reduced length of $w$. We will find $w' \in A_k(\tau)$ with
$\varphi(w') = \varphi(w)$.
Choose $m \in \integers_{\geq 0}$ large enough to ensure that the word
$c_1\dots c_m$ has $k+l$ distinct subwords in $\kerm \varphi$. Call
these subwords $\zeta_i$, $1 \leq i \leq k+l$; and let the remaining
letters form subwords $\chi_i$, $1 \leq i \leq k+l$, so that we have
the decomposition
\[ c_1\dots c_m = \chi_1\zeta_1\chi_2\zeta_2\dots\chi_{k+l}\zeta_{k+l}\]
where each $\varphi(\zeta_i)$ is trivial, and each $\chi_i$ is possibly empty.

Now
\[ \varphi(c_1c_2\dots c_m) = \varphi(\chi_1 \chi_2 \dots \chi_{k+l}) \]
and the reduced length of $\chi_1\chi_2\dots\chi_{k+l}$ is no greater
than $m - l - k$.  Thus the word
$\xi = c_1c_2\dots c_m\chi_{k+ l}^{-1}\dots \chi_2^{-1} \chi_1^{-1}$
is in both $\kerm \varphi$ and $\overline{Ball_{m-l-k}(A, c_1\dots
c_m)}$, and 
\[ \xi w \in \overline{Ball_{m-k}(A, c_1\dots c_m)} \subseteq
A_k(\tau)\] 
and satisfies $\varphi(w) = \varphi(\xi w)$.
\end{proof}


\subsection{Argument for Example \ref{freeandtrivialexample}}
\label{freeandtrivialargument}
In Example \ref{freeandtrivialexample}, $G = B \rtimes_\varphi A$,
where $A = C * D$ for free groups  $C = \langle a_1,\dots, a_n\rangle$, 
$D = \langle a_{n+1}, \dots, a_m\rangle$, and  $\ds B = \oplus_{\omega \in
C} K_\omega$ for
some finitely generated group $K$. The claim is made that any
endpoint of $T = \Gamma(A, \{a_1,a_2,\dots,a_m\})$ represented by a ray $\tau$ whose letters are eventually
selected only from $C$ does not lie in $\Sigma^1$.  Since $\Sigma^1$
is $G$-invariant, we can assume $\tau$ consists of letters entirely in
$C$.  Then $\pi: A \twoheadrightarrow C$ fixes each vertex of $\tau$. Moreover
it makes sense to discuss the subset $C_k(\tau) \subseteq C$.

Let $k \in \integers_{\leq 0}$ be given, and let $S$ be any finite
subset of $B$. We will show that the set $ S' = \{\varphi_a(s) \mid s\in S,\
a \in A_k(\tau)\}$ does not generate $B$. Part
$(ii)$ of  Theorem \ref{Sigma1bygeneratorsthm} thereby ensures that
$\tau(\infty) \nin \Sigma^1(\rho)$.

To show that $S'$ does not generate $B$, we will find an index
$\psi \in C$ such that every $s \in S'$ is trivial at index $\psi$. 

\begin{observation} 
\label{horoprojobs}
If $a \in A$ is in $A_k(\tau)$, then
$\pi(a)$ is in $C_k(\tau)$. 
\end{observation}
\begin{proof}
Since $a \in A_k(\tau)$, there exists $l \geq 0$ (recall $k \leq 0$)
such that $a \in
\overline{\metricball{l-k}{A}{\tau(l)}}$ by Equation \eqref{ahorodef}
on page \pageref{ahorodef}, 
so $\pi(a) \in \overline{\metricball{l-k}{C}{\tau(l)}}$. But this is
contained in $C_k(\tau)$, again by Equation \eqref{ahorodef}.
\end{proof}


Define the set 
\[\mathcal{I}(S) = \{ \omega \in C \mid \exists\, s \in S
\text{ such that } s \text{ is nontrivial at index } \omega\}.\]
Note that $\mathcal{I}(S)$ is a finite set, since $S$ is finite and
each $s \in S$ is nontrivial at only finitely many indices. Define
\[\mathcal{R}(S) = \max\{ \text{reduced length of } \omega \mid
\omega \in \mathcal{I}(S) \}.\]
Since $\mathcal{I}(S)$ is finite, $\mathcal{R}(S)$ is a nonnegative
integer representing the maximum distance (in $C$) from any index of any
nontrivial component of any element of $S$ to the identity
index $1 \in C$.

Since left multiplication by $c \in C$ is an isometry on $C$, it
follows that the maximal distance in $C$ from any nontrivial index of
any element of $\varphi_c(S)$ to $c$ is also $\mathcal{R}(S)$.
Observation \ref{horoprojobs} therefore ensures that the set of
nontrivial indices of elements of $S'$ is a subset of the closed
$\mathcal{R}(S)$-neighborhood of $C_k(\tau)$ in $C$. In fact, this
neighborhood is the set $C_{k-\mathcal{R}(S)}(\tau)$.  This is a
proper subset of $C$ (simply choose any geodesic ray other than $\tau$
and follow it far enough). For any $\psi \in C$ with $\psi \nin
C_{k-\mathcal{R}(S)}(\tau)$, all $s \in S'$ will be trivial at index
$\psi$. So $S'$ can not generate $B$.


\bibliographystyle{plain}
\bibliography{part2-9}

\end{document}